\documentclass[10pt,a4paper]{amsart}
\usepackage{latexsym}
\usepackage{amssymb,amsmath,amsthm,amscd,amsfonts,enumerate}
\usepackage{graphicx}
\usepackage[toc,page]{appendix}
\usepackage{pdflscape}
\usepackage{setspace}
\usepackage{multirow}
\usepackage[usenames]{color}
\usepackage{float}
\usepackage{tikz}
\usepackage{longtable}

\usetikzlibrary{arrows}

\numberwithin{equation}{section}

\theoremstyle{plain}
\newtheorem{theorem} {Theorem} [section]

\newtheorem{lemma} [theorem] {Lemma}

\theoremstyle{definition}

\newcommand \g {\mathfrak{g}}
\renewcommand \a {\mathfrak{a}}

\newcommand \n {\mathfrak{n}}

\newcommand \GL {\operatorname{GL}}

\newcommand \z {\mathfrak{z}}
\newcommand \h {\mathfrak{h}}

\newcommand \C {\mathbb{C}}
\newcommand \R {\mathbb{R}}

\begin{document}
\title{Degenerations of 8-dimensional 2-step nilpotent Lie algebras}
\author{Mar\'ia Alejandra Alvarez}
\address{Departamento de Matem\'aticas - Facultad de Ciencias B\'asicas - Universidad de Antofagasta}
\email{maria.alvarez@uantof.cl}

\thanks{}



\maketitle

\begin{abstract}
In this work, we consider degenerations between 8-dimensional 2-step nilpotent Lie algebras over $\C$ and obtain the geometric classification of the variety $\mathcal{N}^2_8$.
\end{abstract}

\section{Preliminaries}

The algebraic classification of Lie algebras is a {\it wild problem}. Lie algebras are classified up to dimension 6 (see for instance \cite{SW} for a list of indecomposable Lie algebras of dimension $\leq 6$). In the class of nilpotent Lie algebras, there are classifications up to dimension 7 (see for instance \cite{G} or \cite{Ma}). In dimension 8 there are only  classifications of 2-step nilpotent and filiform Lie algebras (see \cite{ZS} and \cite{fili8} respectively). A related problem is the one concerning the geometric classification of Lie algebras, their degenerations, rigid elements and irreducible components. Regarding this problem in the variety of Lie algebras we can mention \cite{CD}, \cite{KN}, \cite{GO}, \cite{S6}, \cite{CPSW}, \cite{W}, \cite{BS}, \cite{B1}, \cite{L}, \cite{B2}, \cite{HNPT}, \cite{NP}, \cite{A1}. Moreover, the study of the geometric classification for varieties of different structures is an active research field, several results have been obtained recently in different directions regarding nilpotent algebras (see for instance \cite{ACK}, \cite{AH}, \cite{AH2}, \cite{FKKV}, \cite{GKK}, \cite{GKKS}, \cite{KKK}, \cite{KPGV}, \cite{KPPV}, \cite{KPV}).

\medskip

In this work we obtain degenerations between 2-step nilpotent Lie algebras of dimension 8 over $\C$ and provide the irreducible components of the variety $\mathcal{N}_8^2$, which turn out to be the orbit closures of three rigid Lie algebras.

\subsection{The variety of Lie algebras}

Let $V$ be a complex $n$-dimensional vector space with a fixed basis
$\left\{e_1,\dots,e_n\right\}$, and let $\g=(V,[\cdot,\cdot])$ be a Lie algebra with underlying vector space $V$ and Lie product $[\cdot,\cdot]$. The set of Lie algebra structures on the space $V$ is an algebraic variety in $\C^{n^3}$ in the following sense: Every Lie algebra structure on $V$, $\g$, can be identified with its set of structure cons\-tants $\left\{c_{ij}^k\right\}\in\C^{n^3}$, where $\displaystyle[e_i,e_j]=\sum_{k=1}^nc_{ij}^ke_k$. This set of structure constants satisfies the polynomial equations given by the skew-symmetry and the Jacobi identity, i.e. $c_{i,j}^k+c_{j,i}^k=0$  and $\displaystyle\sum_{l=1}^n\left(c_{j,k}^lc_{i,l}^r+c_{k,i}^lc_{j,l}^r+c_{i,j}^lc_{k,l}^r\right)=0$. We will denote by $\mathcal{L}_n$ the algebraic variety of Lie algebras of fixed dimension $n$. The group $G=\GL(n,\C)$ acts on $\mathcal{L}_n$ via change of basis: \[g\cdot[X,Y]=g\left([g^{-1}X,g^{-1}Y]\right),\quad X,Y\in\g,\ g\in \GL(n,\C).\]
Also, one can define the Zariski topology on $\mathcal{L}_n$

\medskip

Given two Lie algebras $\g$ and $\h$, we say that {\bf $\g$ degenerates to $\h$}, and denoted by $\g\to\h$, if $\h$ lies in the Zariski closure of the $G$-orbit $O(\g)$. An element $\g\in\mathcal{L}_n$ is called {\bf rigid}, if its orbit $O(\g)$ is open in $\mathcal{L}_n$. Since each orbit $O(\g)$ is a constructible set, its closures relative to the Euclidean and the Zariski topologies are the same (see \cite{M}, 1.10 Corollary 1, p. 84). As a consequence the following is obtained:

\begin{lemma}
Let $\C(t)$ be the field of fractions of the polynomial ring $\C[t]$. If there
exists an operator $g_t\in\GL(n,\C(t))$ such that $\displaystyle\lim_{t\to 0}g_t\cdot\g=\h$, then $\g\to\h$.
\end{lemma}

\medskip

For example, every $n$-dimensional Lie algebra degenerates to the abelian one. The operator defined by $g_t(e_k)=t^{-1}e_k$ for $1\leq k\leq n$, gives us
\[\lim_{t\to0}g_t\cdot[e_i,e_j]=\lim_{t\to0}g_t\left([g_t^{-1}(e_i),g_t^{-1}(e_j)]\right)=\lim_{t\to0}g_t\left([te_i,te_j]\right)=\lim_{t\to0}t^2\sum_{k=1}^n c_{i,j}^kg_t(e_k)=\lim_{t\to0}t\left(\sum_{k=1}^n c_{i,j}^ke_k\right)=0.\]

\subsection{The variety of 2-step nilpotent Lie algebras}

The variety $\mathcal{N}^2_n$, is the closed subset of $\mathcal{L}_n$ given by all at most 2-step nilpotent Lie algebras, i.e. those that satisfying $[\g,[\g,\g]]=0$.

\medskip

The geometric classification of the varieties $\mathcal{N}^2_n$ for $n\leq 7$ can been seen from \cite{A1}, and now can be recuperated from this work.

\medskip

Here we consider the classification of indecomposable 8-dimensional 2-step nilpotent Lie algebras over $\C$ obtained by Zaili and Shaoqiang in \cite{ZS}; the classification of indecomposable nilpotent Lie algebras of dimension 7 given by Gong in \cite{G}, and the classification of Lie algebras of dimension $\leq 6$ from \cite{SW}.

\begin{theorem}
The isomorphism classes of Lie algebras in $\mathcal{N}^2_8$ are:
\begin{center}

\begin{spacing}{1.2}
\begin{longtable}{|c|l|c|}
\caption[]{Lie algebras in $\mathcal{N}^2_8$}
\label{Table1}\\
\hline
$\g$ & \multicolumn{1}{|c|}{Lie Product} & $\dim O(\g)$\\
\hline
\endfirsthead
\caption[]{(continued)}\\
\hline
$\g$ & \multicolumn{1}{|c|}{Lie Product} & $\dim O(\g)$\\
\hline
\endhead
$N_1^{8,2}$ & $[x_1,x_2]=x_7$, $[x_3,x_4]=x_8$, $[x_5,x_6]=x_7+x_8$ & $42$ \\ \hline
$N_2^{8,2}$ & $[x_1,x_2]=x_7$, $[x_4,x_5]=x_7$, $[x_1,x_3]=x_8$, $[x_4,x_6]=x_8$ & $38$\\ \hline
$N_3^{8,2}$ & $[x_1,x_2]=x_7$, $[x_4,x_5]=x_7$, $[x_3,x_4]=x_8$, $[x_5,x_6]=x_8$ & $41$\\ \hline
$N_4^{8,2}$ & $[x_1,x_2]=x_7$, $[x_3,x_4]=x_7$, $[x_5,x_6]=x_7$, $[x_4,x_5]=x_8$ & $36$\\ \hline
$N_5^{8,2}$ & $[x_1,x_2]=x_7$, $[x_3,x_4]=x_7$, $[x_5,x_6]=x_7$, $[x_2,x_3]=x_8$, $[x_4,x_5]=x_8$ & $40$ \\ \hline
$N_1^{8,3}$ & $[x_1,x_2]=x_6$, $[x_4,x_5]=x_6$, $[x_2,x_3]=x_7$, $[x_1,x_3]=x_8$ & $39$ \\ \hline
$N_2^{8,3}$ & $[x_1,x_2]=x_6$, $[x_4,x_5]=x_6$, $[x_2,x_3]=x_7$, $[x_3,x_4]=x_8$ & $38$ \\ \hline
$N_3^{8,3}$ & $[x_1,x_2]=x_6$, $[x_2,x_3]=x_7$, $[x_4,x_5]=x_7$, $[x_3,x_4]=x_8$ & $41$ \\ \hline
$N_4^{8,3}$ & $[x_1,x_2]=x_6$, $[x_2,x_3]=x_7$, $[x_4,x_5]=x_7$, $[x_3,x_4]=x_8$, $[x_5,x_1]=x_8$ & $44$ \\ \hline
$N_5^{8,3}$ & $[x_1,x_2]=x_6$, $[x_1,x_3]=x_7$, $[x_1,x_4]=x_8$, $[x_2,x_5]=x_7$ & $37$ \\ \hline
$N_6^{8,3}$ & $[x_1,x_2]=x_6$, $[x_1,x_3]=x_7$, $[x_1,x_4]=x_8$, $[x_2,x_3]=x_8$, $[x_4,x_5]=x_7$ & $42$ \\ \hline
$N_7^{8,3}$ & $[x_1,x_2]=x_6$, $[x_1,x_3]=x_7$, $[x_1,x_5]=x_8$, $[x_2,x_4]=x_8$, $[x_3,x_4]=x_6$ & $40$ \\ \hline
$N_8^{8,3}$ & $[x_1,x_2]=x_6$, $[x_1,x_3]=x_7$, $[x_2,x_3]=x_8$, $[x_1,x_4]=x_8$, $[x_2,x_5]=x_7$ & $38$ \\ \hline
$N_9^{8,3}$ & $[x_1,x_2]=x_6$, $[x_1,x_3]=x_7$, $[x_2,x_3]=x_8$, $[x_1,x_4]=x_8$, $[x_2,x_5]=x_7$, $[x_4,x_5]=x_6$ & $45$ \\ \hline
$N_{10}^{8,3}$ & $[x_1,x_2]=x_6$, $[x_2,x_3]=x_7$, $[x_3,x_4]=x_7$, $[x_4,x_5]=x_8$ & $42$ \\ \hline
$N_{11}^{8,3}$ & $[x_1,x_2]=x_6$, $[x_2,x_3]=x_7$, $[x_3,x_4]=x_8$, $[x_4,x_5]=x_7$, $[x_5,x_1]=x_7$ & $43$ \\ \hline
$N_1^{8,4}$ & $[x_1,x_2]=x_5$, $[x_2,x_3]=x_6$, $[x_3,x_4]=x_7$, $[x_4,x_1]=x_8$ & $40$ \\ \hline
$N_2^{8,4}$ & $[x_1,x_2]=x_5$, $[x_1,x_3]=x_6$, $[x_2,x_3]=x_7$, $[x_1,x_4]=x_8$ & $37$ \\ \hline
$N_3^{8,4}$ & $[x_1,x_2]=x_5$, $[x_1,x_3]=x_6$, $[x_2,x_4]=x_6$, $[x_2,x_3]=x_7$, $[x_1,x_4]=x_8$ & $39$\\ \hline\hline
$(17)$ & $[x_1,x_2]=x_7$, $[x_3,x_4]=x_7$, $[x_5,x_6]=x_7$ & $28$ \\ \hline
$(27A)$ & $[x_1,x_2]=x_6$, $[x_1,x_4]=x_7$, $[x_3,x_5]=x_7$ & $35$\\ \hline
$(27B)$ & $[x_1,x_2]=x_6$, $[x_1,x_5]=x_7$, $[x_3,x_4]=x_6$, $[x_2,x_3]=x_7$ & $37$ \\ \hline
$(37A)$ & $[x_1,x_2]=x_5$, $[x_2,x_3]=x_6$, $[x_2,x_4]=x_7$ & $31$\\ \hline
$(37B)$ & $[x_1,x_2]=x_5$, $[x_2,x_3]=x_6$, $[x_3,x_4]=x_7$ & $36$ \\ \hline
$(37C)$ & $[x_1,x_2]=x_5$, $[x_2,x_3]=x_6$, $[x_2,x_4]=x_7$, $[x_3,x_4]=x_5$ & $34$\\ \hline
$(37D)$ & $[x_1,x_2]=x_5$, $[x_1,x_3]=x_6$, $[x_2,x_4]=x_7$, $[x_3,x_4]=x_5$ & $37$ \\ \hline\hline
$\n_{6,1}$ & $[x_4,x_5]=x_2$, $[x_4,x_6]=x_3$, $[x_5,x_6]=x_1$ & $30$ \\ \hline
$\n_{6,2}$ & $[x_3,x_6]=x_1$, $[x_5,x_4]=x_1$, $[x_4,x_6]=x_2$ & $31$\\ \hline\hline
$\n_{5,1}\oplus\n_{3,1}$ & $[x_3,x_5]=x_1$, $[x_4,x_5]=x_2$, $[x_7,x_8]=x_6$ & $38$\\ \hline
$\n_{5,1}$ & $[x_3,x_5]=x_1$, $[x_4,x_5]=x_2$ & $27$\\ \hline
$\n_{5,3}\oplus\n_{3,1}$ & $[x_2,x_4]=x_1$, $[x_3,x_5]=x_1$, $[x_7,x_8]=x_6$ & $37$\\ \hline 
$\n_{5,3}$ & $[x_2,x_4]=x_1$, $[x_3,x_5]=x_1$ & $25$\\ \hline\hline
$\n_{3,1}\oplus\n_{3,1}$ & $[x_2,x_3]=x_1$, $[x_5,x_6]=x_4$ & $32$\\ \hline
$\n_{3,1}$ & $[x_2,x_3]=x_1$ & $18$ \\ \hline\hline
$\C^8$ & $[\cdot,\cdot]=0$ & $0$\\ \hline
\hline

\end{longtable}
\end{spacing}
\end{center}
\end{theorem}

\section{Invariants for Lie algebras and non-degenerations}

\medskip

The techniques used in this work are the usual for obtaining non-degenerations arguments. There are several invariants for the orbit closure of a Lie algebra that have been successfully applied in previous works. Among them, we mention the ones that we use in this work.

\medskip

\begin{lemma}\label{lemma:rel}
Let $\g,\h\in\mathcal{L}_n$. If $\g\to\h$, then the following relations must hold:
\begin{enumerate}[(a)]
\item\label{inv:orb} $\dim\ O(\g)>\dim\ O(\h)$.
\item\label{inv:center} $\dim\ \z(\g)\leq\dim\ \z(\h)$, where $\z(\g)$ is the center of $\g$.
\item\label{inv:der} $\dim\ [\g,\g]\geq\dim\ [\h,\h]$.
\item\label{inv:coh} $\dim\  H^k(\g)\leq\dim\ H^k(\h)$ for $0\leq k\leq n$, where $H^k(\g)$ is the $k$-th trivial cohomology group for $\g$.
\item\label{inv:ab} $\a(\g)\leq\a(\h)$, where $\a(\g)=\max\{\dim W:\ W\text{ is an abelian subalgebra of }\g\}$.
\end{enumerate}
\end{lemma}

\begin{proof}
The first relation follows from the Closed Orbit Lemma (see \cite{B}, I. Lemma 1.8, p. 53). Items (b) and (c) follow by proving that the corresponding sets are closed (see \cite{C}, $\S$3 Theorem 2, p. 14). The proofs of $(d)$ and $(e)$ can be found in \cite{B2} and \cite{GO}, respectively.
\end{proof}

Using the previous Lemma we obtain a series of non-degenerations that are summarized in Table \ref{Table1}. We explain the arguments with an example: $N_1^{8,3}\not\to (37D)$ because $\dim H^4(N_1^{8,3})=30>28=\dim H^4((37D))$ contradicting Lemma \ref{lemma:rel} (\ref{inv:coh}).
\begin{landscape}

\begin{center}
\begin{table}
\begin{spacing}{1.25}
\caption{Non-degenerations}
\label{Table1}
\begin{tabular}{|l|l|}
\hline
\multicolumn{1}{|c|}{$\g\not\to\h$} & Reason\\ \hline

$\n_{5,1}\not\to\n_{5,3};\quad \n_{6,1}\not\to(17),\n_{5,3};\quad (37A)\not\to(17);\quad \n_{6,2}\not\to(17);\quad \n_{3,1}\oplus\n_{3,1}\not\to(17)$ & \multirow{11}{*}{Lemma \ref{lemma:rel} (\ref{inv:center})}\\ 

$(37C)\not\to(17);\quad (27A)\not\to(17);\quad (37B)\not\to(27A),(17);\quad (27B)\not\to N_4^{8,2},(17)$ &  \\ 

$N_5^{8,3}\not\to N_4^{8,2},(17);\quad (37D)\not\to N_4^{8,2},(27A),(17);\quad N_2^{8,4}\not\to N_4^{8,2},(17),(27A)$ &  \\ 

$\n_{5,3}\oplus\n_{3,1}\not\to N_4^{8,2},(17);\quad N_2^{8,3}\not\to\n_{5,3}\oplus\n_{3,1},N_4^{8,2},(17);\quad N_8^{8,3}\not\to\n_{5,3}\oplus\n_{3,1},N_4^{8,2},(17)$ &  \\ 

$\n_{5,1}\oplus\n_{3,1}\not\to N_4^{8,2},(17),\n_{5,3}\oplus\n_{3,1};\quad N_1^{8,3}\not\to N_2^{8,2},N_4^{8,2},(17);\quad N_1^{8,3}\not\to\n_{5,3}\oplus\n_{3,1}$ & \\ 

$N_3^{8,4}\not\to N_2^{8,2},N_2^{8,3},N_8^{8,3},\n_{5,1}\oplus\n_{3,1},(27B),N_5^{8,3},\n_{5,3}\oplus\n_{3,1},N_4^{8,2},(27A),(17)$ &  \\ 

$N_7^{8,3}\not\to N_2^{8,2},N_4^{8,2},(17);\quad N_7^{8,3}\not\to\n_{5,3}\oplus\n_{3,1};\quad N_3^{8,3}\not\to N_5^{8,2},N_2^{8,2},\n_{5,3}\oplus\n_{3,1},N_4^{8,2},(17)$ &  \\ 

$N_1^{8,4}\not\to N_1^{8,3},N_2^{8,2},N_2^{8,3},N_8^{8,3},\n_{5,1}\oplus\n_{3,1},(27B),N_5^{8,3},\n_{5,3}\oplus\n_{3,1},N_4^{8,2},(27A),(17)$ &  \\ 

$N_6^{8,3}\not\to N_3^{8,2},N_5^{8,2},N_2^{8,2},N_4^{8,2},(17),\n_{5,3}\oplus\n_{3,1};\quad N_{11}^{8,3}\not\to N_1^{8,2},N_3^{8,2},N_5^{8,2},N_2^{8,2},N_4^{8,2},(17),\n_{5,3}\oplus\n_{3,1}$ &  \\ 

$N_{10}^{8,3}\not\to N_3^{8,2},N_5^{8,2},N_2^{8,2},N_4^{8,2},(17),\n_{5,3}\oplus\n_{3,1};\quad N_{4}^{8,3}\not\to N_1^{8,2},N_3^{8,2},N_5^{8,2},N_2^{8,2},N_4^{8,2},(17),\n_{5,3}\oplus\n_{3,1}$ &  \\ 

$N_{9}^{8,3}\not\to N_1^{8,2},N_3^{8,2},N_5^{8,2},N_2^{8,2},(17),N_4^{8,2},\n_{5,3}\oplus\n_{3,1}$ &  \\ \hline\hline

$(17)\not\to\n_{5,1};\quad \n_{6,2}\not\to\n_{6,1};\quad \n_{3,1}\oplus\n_{3,1}\not\to(37A),\n_{6,1};\quad (27A)\not\to(37C),(37A),\n_{6,1}$ & \multirow{9}{*}{Lemma \ref{lemma:rel} (\ref{inv:der})}\\ 

$N_4^{8,2}\not\to(37C),(37A),\n_{6,1};\quad (27B)\not\to (37B),(37C),(37A),\n_{6,1};\quad \n_{5,3}\oplus\n_{3,1}\not\to (37B),(37C),(37A),\n_{6,1}$ & \\ 

$ N_2^{8,2}\not\to N_5^{8,3},(37D),N_2^{8,4},(37B),(37C),(37A),\n_{6,1};\quad N_2^{8,3}\not\to N_2^{8,4};\quad N_8^{8,3}\not\to N_2^{8,4}$ &  \\

$\n_{5,1}\oplus\n_{3,1}\not\to N_2^{8,4};\quad N_1^{8,3}\not\to N_2^{8,4};\quad N_3^{8,2}\not\to N_7^{8,3},N_1^{8,4},N_1^{8,3},N_3^{8,4},N_2^{8,3},N_8^{8,3},\n_{5,1}\oplus\n_{3,1}$ & \\ 

$N_5^{8,2}\not\to N_1^{8,3},N_3^{8,4},N_2^{8,3},N_8^{8,3},\n_{5,1}\oplus\n_{3,1},N_5^{8,3},(37D),N_2^{8,4},(37B),(37C),(37A),\n_{6,1}$ &  \\ 

$N_3^{8,2}\not\to N_5^{8,3},(37D),N_2^{8,4},(37B),(37C),(37A),\n_{6,1};\quad N_3^{8,3}\not\to N_1^{8,4},N_3^{8,4},N_2^{8,4}$ &  \\ 

$N_1^{8,2}\not\to N_3^{8,3},N_7^{8,3},N_1^{8,4},N_1^{8,3},N_3^{8,4},N_2^{8,3},N_8^{8,3},\n_{5,1}\oplus\n_{3,1},N_5^{8,3},(37D),N_2^{8,4},(37B),(37C),(37A),\n_{6,1}$ &   \\ 

$N_7^{8,3}\not\to N_3^{8,4},N_2^{8,4};\quad N_6^{8,3}\not\to N_1^{8,4},N_3^{8,4},N_2^{8,4};\quad N_{11}^{8,3}\not\to N_1^{8,4},N_3^{8,4},N_2^{8,4}$ &  \\ 

$N_{10}^{8,3}\not\to N_1^{8,4},N_3^{8,4},N_2^{8,4};\quad N_{4}^{8,3}\not\to N_1^{8,4},N_3^{8,4},N_2^{8,4};\quad N_{9}^{8,3}\not\to N_{4}^{8,3},N_1^{8,4},N_3^{8,4},N_2^{8,4}$ &  \\ \hline\hline

$(37A)\not\to\n_{6,1};\quad N_7^{8,3}\not\to N_1^{8,3};\quad N_3^{8,3}\not\to N_1^{8,3};\quad N_{10}^{8,3}\not\to N_1^{8,3}$ & Lemma \ref{lemma:rel} (\ref{inv:coh}) $k=2$ \\ \hline\hline

$\n_{5,1}\oplus\n_{3,1}\not\to (27B);\quad N_1^{8,2}\not\to N_4^{8,2}$ &  Lemma \ref{lemma:rel} (\ref{inv:coh}) $k=3$ \\ \hline\hline

$N_2^{8,2}\not\to N_4^{8,2},(17);\quad N_2^{8,3}\not\to(37D);\quad \n_{5,1}\oplus\n_{3,1}\not\to(37D);\quad N_1^{8,3}\not\to(37D)$ & Lemma \ref{lemma:rel} (\ref{inv:coh}) $k=4$\\ \hline\hline

$(37A)\not\to\n_{5,3};\quad N_2^{8,2}\not\to\n_{5,3}\oplus\n_{3,1}$ & Lemma \ref{lemma:rel} (\ref{inv:ab}) \\ \hline\hline

$N_1^{8,3}\not\to N_8^{8,3}$ & $N_1^{8,3}\not\to (37D)$ \\ \hline
\hline
\end{tabular}
\end{spacing}
\end{table}

\end{center}

\end{landscape}

\normalsize

We continue in proving the non-existence of degenerations. The next Lemma can be found, for instance, in \cite{GO} or \cite{S6}:

\begin{lemma}
Let $B$ be a Borel subgroup of $\GL(n,\C)$ and let $S\in\mathcal{L}_n$ be a closed subset which is $B$-stable. If $\g\to\h$ and $\g\in S$ then there exists a Lie algebra $\overline{\h}\in S$ such that $\overline{\h}\simeq\h$. 
\end{lemma}

\medskip

Consider now the following sets:
\begin{align*}
S_1=&\left\{\g=(c_{ij}^k)\in\mathcal{L}_8\ \left|\ 
\begin{array}{ll}
c_{r6}^7=0, &\text{ for }1\leq r\leq 5\\
c_{rs}^t=0, &\text{ for }7\leq s\leq 8,\ 1\leq r<s\\
\end{array}\right.\right\},\\
S_2=&\left\{\g=(c_{ij}^k)\in\mathcal{L}_8\ \left|\ \begin{array}{ll}c_{rs}^7=\lambda c_{rs}^6, & \text{for } 4\leq s\leq 5,\ 1\leq r<s\\
c_{rs}^8=\mu c_{r4}^6, & \text{for } 4\leq s\leq 5,\ 1\leq r<s\\
c_{rs}^t=0, &\text{ for }6\leq s\leq 8,\ 1\leq r<s\\
\lambda,\mu\in\C. & 
\end{array}\right.\right\},\\
S_3=&\left\{\g=(c_{ij}^k)\in\mathcal{L}_8\ \left|\ \begin{array}{ll}c_{rs}^6=0, & \text{for } 3\leq s\leq 5,\ 1\leq r<s\\
c_{r5}^8=\lambda c_{r5}^7, & \text{for } 1\leq r\leq 4\\
c_{r5}^7=\mu_r c_{45}^7, & \text{for } 1\leq r\leq 3\\
c_{rs}^t=0, &\text{ for }6\leq s\leq 8,\ 1\leq r<s\\
\lambda,\mu_1,\mu_2,\mu_3\in\C. & 
\end{array}\right.\right\},\\
S_4=&\left\{\g=(c_{ij}^k)\in\mathcal{L}_8\ \left|\ \begin{array}{ll}c_{rs}^6=0, & \text{for } 3\leq s\leq 5,\ 1\leq r<s\\
c_{rs}^7=0, & \text{for } 3\leq s\leq 5,\ 2\leq r<s\\
c_{1r}^7=0, & \text{for } 4\leq r\leq 5\\
c_{r5}^8=\mu_r c_{45}^8, & \text{for } 1\leq r\leq 3\\
c_{rs}^t=0, &\text{ for }6\leq s\leq 8,\ 1\leq r<s\\
\mu_1,\mu_2,\mu_3\in\C. & 
\end{array}\right.\right\}.
\end{align*}

\medskip

It is not difficult to check that all these sets are $B$-stable. Then we obtain the following results:

\begin{enumerate}
\item $N_3^{8,2}\in S_1$ but there is no $g\in\GL(8,\C)$ such that $g\cdot N_5^{8,2}\in S_1$, therefore $N_3^{8,2}\not\to N_5^{8,2}$.

\medskip

\item $N_1^{8,3}\in S_2$ but there is no $g\in\GL(8,\C)$ such that $g\cdot (27B)\in S_2$, therefore $N_1^{8,3}\not\to (27B)$. This also implies that $N_2^{8,3}\not\to (27B)$. 

\medskip

\item $N_3^{8,3}\in S_3$ but there is no $g\in\GL(8,\C)$ such that $g\cdot N_7^{8,3}\in S_3$, therefore $N_3^{8,3}\not\to N_7^{8,3}$.

\medskip

\item $\n_{5,1}\oplus\n_{3,1}\in S_4$ but there is no $g\in\GL(8,\C)$ such that $g\cdot N_5^{8,3}\in S_4$, therefore $\n_{5,1}\oplus\n_{3,1}\not\to N_5^{8,3}$.
\end{enumerate}

\bigskip

Finally, we want to prove the following lemma:

\begin{lemma}\label{lem:ultdeg}
$\n_{5,1}\oplus\n_{3,1}$ is not in the orbit closure of $N_7^{8,3}$.
\end{lemma}

In order to do this, we use the results in \cite{GO2}. First, denote by $\g(b)$ the central extension of $\g$ by $\C^r$ defined by the 2-cocycle $b$ and let $\mathbb{B}_\g^r=\{b\in Z^2(\g,\C^r):\ b^\perp\cap\z(\g)=0\}$. Then one has:

\begin{theorem}[\cite{GO2}, Theorem 2.2]\label{thm:orbit_coc}
For $b_0, b_1\in\mathbb{B}_\g^r$, $b_0\in\overline{O(b_1)}$ if and only if $\g(b_0)\in\overline{O(\g(b_1))}$.
\end{theorem}

\medskip

With this result we obtain the last non-degeneration:

\medskip

\begin{proof}[Proof of Lemma \ref{lem:ultdeg}]
A base change allow us to write the following products:

\[\begin{tabular}{c|c|c}
$\n_{5,1}\oplus\n_{3,1}$ & $N_7^{8,3}$ & $\g=\h_3\oplus\C^3$\\ \hline
$[e_1,e_2]=e_3$ & $[e_1,e_2]=e_3$ & $[e_1,e_2]=e_3$\\
$[e_1,e_4]=e_7$ & $[e_1,e_4]=e_7$ & \\
$[e_5,e_6]=e_8$ & $[e_1,e_5]=e_8$ & \\
 & $[e_2,e_6]=e_7$ & \\
  & $[e_4,e_6]=e_8$ & \\
\end{tabular}\]

It is clear that $\n_{5,1}\oplus\n_{3,1}=\g(b_0)$ and $N_7^{8,3}=\g(b_1)$ where

\[\begin{tabular}{c|c}
$b_0$ & $b_1$ \\ \hline
$b_0(e_1,e_4)=e_7$ & $b_1(e_1,e_4)=e_7$ \\
$b_0(e_5,e_6)=e_8$ & $b_1(e_2,e_6)=e_7$ \\
 & $b_1(e_2,e_5)=e_8$ \\
 & $b_1(e_4,e_6)=e_8$ \\
\end{tabular}\]

\medskip

By Theorem \ref{thm:orbit_coc}, $\n_{5,1}\oplus\n_{3,1}$ is in the orbit closure of $N_7^{8,3}$ if and only if $b_0$ is in the orbit closure of $b_1$.

\medskip

Suppose $b_0$ is in the orbit closure of $b_1$. Then by Theorem 1.2 of \cite{GO2}, there is a coordinate ring $\C[Z]$ for some affine set $Z$, an element $g\in \GL(8,\C(Z))$, and an element $x\in Z$ such that $b_0$ is the evaluation of $g\cdot b_1$ at $x$. The element $g$ is of the form
\[g=\begin{pmatrix}
\alpha & 0\\
u & \Phi
\end{pmatrix},\quad\text{where }\alpha^{-1}=\begin{pmatrix}
a_{11} & a_{12} & 0 & 0 & 0 & 0\\
a_{21} & a_{22} & 0 & 0 & 0 & 0\\
a_{31} & a_{32} & a_{33} & a_{34} & a_{35} & a_{36}\\
a_{41} & a_{42} & 0 & a_{44} & a_{45} & a_{46}\\
a_{51} & a_{52} & 0 & a_{54} & a_{55} & a_{56}\\
a_{61} & a_{62} & 0 & a_{64} & a_{65} & a_{66}
\end{pmatrix}\quad\text{and }\Phi=\begin{pmatrix}
p & r\\
q & s
\end{pmatrix},\]
($\alpha\in\operatorname{Aut}(\g)$). Then we have
\begin{align*}
g\cdot b_1(e_5,e_6)=&(a_{45}a_{66}-a_{65}a_{46})(re_7+se_8),\\
g\cdot b_1(e_4,e_5)=&(a_{44}a_{65}-a_{64}a_{45})(re_7+se_8),\\
g\cdot b_1(e_4,e_6)=&(a_{44}a_{66}-a_{64}a_{46})(re_7+se_8),\\
g\cdot b_1(e_1,e_4)=&(a_{11}a_{44}+a_{21}a_{64})(pe_7+qe_8)+(a_{11}a_{54}+a_{41}a_{64})(re_7+se_8).
\end{align*}

If $g\cdot b_1$ evaluated in $x$ is $b_0$ we obtain:
\begin{equation}\label{eq1}
e_8=(a_{45}(x)a_{66}(x)-a_{65}(x)a_{46}(x))(r(x)e_7+s(x)e_8),
\end{equation}

therefore $r(x)=0$ and $s(x)\neq0$. Moreover,
\begin{align}
0=&a_{44}(x)a_{65}(x)-a_{64}(x)a_{45}(x),\label{eq2}\\
0=&a_{44}(x)a_{66}(x)-a_{64}(x)a_{46}(x),\label{eq3}\\
1=&(a_{11}(x)a_{44}(x)+a_{21}(x)a_{64}(x))p(x).\label{eq4}
\end{align}

Equations (\ref{eq2}) and (\ref{eq3}) imply that
\begin{itemize}
\item If $a_{44}(x)\neq0$ then $a_{65}(x)=\displaystyle\frac{a_{64}(x)a_{45}(x)}{a_{44}(x)}$ and $a_{66}(x)=\displaystyle\frac{a_{64}(x)a_{46}(x)}{a_{44}(x)}$. This contradicts equation \eqref{eq1}.
\item If $a_{44}(x)=a_{64}(x)=0$ this contradicts equation \eqref{eq4}.
\item If $a_{44}(x)=a_{45}(x)=a_{46}(x)=0$ this contradicts equation \eqref{eq1}.
\end{itemize}

Therefore $\n_{5,1}\oplus\n_{3,1}\not\in\overline{O(N_7^{8,3})}$.

\end{proof}

\section{Degenerations}

Next, we list in Table \ref{Table2} all the primary degenerations (those that cannot be obtained by transitivity). 

\medskip

Consider for example the Lie algebra $N_5^{8,2}$ and the base change
\[x_1=e_8,\quad x_2=e_2,\quad x_3=e_4,\quad x_4=e_3,\quad x_5=e_5,\quad x_6=\frac{1}{t}e_7,\quad x_7=e_1+\frac{1}{t}e_5,\quad x_8=e_6.\]

Then the new product is given by:
\begin{align*}
[x_2,x_4]=&x_1,&[x_2,x_7]=&-tx_6,& [x_3,x_4]=&-tx_6,&[x_3,x_5]=&x_1,\\
[x_3,x_7]=&tx_1,& [x_5,x_8]=&tx_6,& [x_7,x_8]=&x_6.&&
\end{align*}

When $t\to 0$ we obtain the Lie product of $\n_{5,3}\oplus\n_{3,1}$, therefore $N_5^{8,2}\to\n_{5,3}\oplus\n_{3,1}$.

\scriptsize
{
\begin{spacing}{1.2}
\begin{longtable}{|c|llll|}
\caption[]{Degenerations}
\label{Table2}\\
\hline
$\g\to\h$ & \multicolumn{4}{|c|}{Parameterized basis} \\
\hline
\endfirsthead
\caption[]{(continued)}\\
\hline
$\g\to\h$ & \multicolumn{4}{|c|}{Parameterized basis} \\
\hline
\endhead

\multirow{2}{*}{$\n_{5,3}\to\n_{3,1}$} & $x_1=e_1$, & $x_2=e_2$, & $x_3=e_4$, & $x_4=e_3$, \\
& $x_5=te_5$, & $x_6=e_6$ & $x_7=e_7$ & $x_8=e_8$.\\ \hline

\multirow{2}{*}{$\n_{5,1}\to\n_{3,1}$} & $x_1=e_1$, & $x_2=e_3$, & $x_3=e_5$, & $x_4=te_4$,\\
 & $x_5=e_2$, & $x_6=e_6$ & $x_7=e_7$ & $x_8=e_8$.\\ \hline

\multirow{2}{*}{$(17)\to\n_{5,3}$} & $x_1=e_7$, & $x_2=e_1$, & $x_3=e_3$, & $x_4=e_2$, \\
& $x_5=e_4$, & $x_6=te_6$, & $x_7=e_5$ & $x_8=e_8$.\\ \hline

\multirow{2}{*}{$\n_{6,1}\to\n_{5,1}$} & $x_1=e_3$, & $x_2=e_1$, & $x_3=e_4$, & $x_4=e_5$,\\
 & $x_5=e_6$, & $x_6=\frac{1}{t}e_2$ & $x_7=e_7$, & $x_8=e_8$.\\ \hline

\multirow{2}{*}{$(37A)\to\n_{5,1}$} & $x_1=e_6$, & $x_2=e_7$, & $x_3=-e_3$, & $x_4=-e_4$,\\
 & $x_5=e_2$, & $x_6=te_1$, & $x_7=e_5$ & $x_8=e_8$.\\ \hline

\multirow{2}{*}{$\n_{6,2}\to\n_{5,1}$} & $x_1=e_1$, & $x_2=e_2$, & $x_3=e_3$, & $x_4=e_4$,\\
 & $x_5=e_6$, & $x_6=te_5$ & $x_7=e_7$, & $x_8=e_8$. \\ \hline

\multirow{2}{*}{$\n_{6,2}\to\n_{5,3}$} & $x_1=e_1$, & $x_2=e_3$, & $x_3=e_5$, & $x_4=e_6$,\\
 & $x_5=e_4$, & $x_6=\frac{1}{t}e_2$ & $x_7=e_7$, & $x_8=e_8$.\\ \hline

\multirow{2}{*}{$\n_{3,1}\oplus\n_{3,1}\to\n_{6,2}$} & $x_1=-\frac{1}{t}e_4$, & $x_2=e_1+\frac{1}{t^{2}}e_4$, & $x_3=e_6$, & $x_4=e_2-\frac{1}{t}e_6$,\\
 &  $x_5=e_5$ & $x_6=\frac{1}{t}e_5+e_3$, & $x_7=e_7$, & $x_8=e_8$.\\ \hline

\multirow{2}{*}{$(37C)\to\n_{3,1}\oplus\n_{3,1}$} & $x_1=e_5$, & $x_2=\frac{1}{t}e_1$, & $x_3=te_2$, & $x_4=t^2e_7$, \\
& $x_5=te_3+t^2e_2$, & $x_6=e_4$, & $x_7=e_6$, & $x_8=e_8$.\\ \hline

\multirow{2}{*}{$(37C)\to(37A)$} & $x_1=e_1$, & $x_2=e_2$, & $x_3=te_3$, & $x_4=e_4$,\\
 & $x_5=e_5$, & $x_6=te_6$, & $x_7=e_7$, & $x_8=e_8$.\\ \hline

\multirow{2}{*}{$(37C)\to\n_{6,1}$} & $x_1=e_5$, & $x_2=e_6$, & $x_3=e_7$, & $x_4=e_2$, \\
& $x_5=e_3$, & $x_6=e_4$, & $x_7=te_1$, & $x_8=e_8$.\\ \hline

\multirow{2}{*}{$(27A)\to\n_{3,1}\oplus\n_{3,1}$} & $x_1=e_6$, & $x_2=e_1$, & $x_3=e_2$, & $x_4=e_7$, \\
& $x_5=e_3$, & $x_6=e_5$, & $x_7=te_4$, & $x_8=e_8$.\\ \hline

\multirow{2}{*}{$N_4^{8,2}\to (27A)$} & $x_1=e_4$, & $x_2=e_5$, & $x_3=e_1$, & $x_4=-e_3$, \\
& $x_5=e_2$, & $x_6=e_8$, & $x_7=e_7$, &  $x_8=te_6$.\\ \hline

\multirow{2}{*}{$N_4^{8,2}\to(17)$} & $x_1=e_1$, & $x_2=e_2$, & $x_3=e_3$, & $x_4=e_4$, \\
& $x_5=e_5$, & $x_6=e_6$, & $x_7=e_7$, & $x_8=\frac{1}{t}e_8$.\\ \hline

\multirow{2}{*}{$(37B)\to(37C)$} & $x_1=e_3$, & $x_2=e_1+i\frac{1}{\sqrt{t}}e_2+i\frac{1}{\sqrt{t^3}}e_3-\frac{1}{t}e_4$, & $x_3=-\frac{1}{\sqrt{t}}e_3+e_4$, & $x_4=e_2+\frac{1}{t}e_3$, \\
& $x_5=-\frac{1}{t}e_7$, & $x_6=\frac{1}{t}e_6$, & $x_7=e_5+\frac{1}{t^2}e_7$, & $x_8=e_8$.\\ \hline

\multirow{2}{*}{$(27B)\to(27A)$} & $x_1=e_3$, & $x_2=e_4$, & $x_3=te_1$, & $x_4=-e_2$, \\
& $x_5=\frac{1}{t}e_5$, & $x_6=e_6$, & $x_7=e_7$, & $x_8=e_8$.\\ \hline

\multirow{2}{*}{$N_5^{8,3}\to(37B)$} & $x_1=-e_4$, & $x_2=e_1$, & $x_3=e_2$, & $x_4=e_5$, \\
& $x_5=e_8$, & $x_6=e_6$, & $x_7=e_7$, & $x_8=te_3$. \\ \hline

\multirow{2}{*}{$N_5^{8,3}\to(27A)$} & $x_1=e_1$, & $x_2=e_4$, & $x_3=e_2$, & $x_4=e_3$, \\
& $x_5=e_5$, & $x_6=e_8$, & $x_7=e_7$, & $x_8=\frac{1}{t}e_6$. \\ \hline

\multirow{2}{*}{$(37D)\to(37B)$} & $x_1=e_2$, & $x_2=e_4$, & $x_3=-e_3$, & $x_4=te_1$, \\
& $x_5=e_7$, & $x_6=e_5$, & $x_7=te_6$, & $x_8=e_8$.\\ \hline

\multirow{2}{*}{$N_2^{8,4}\to (37B)$} & $x_1=e_3$, & $x_2=e_2$, & $x_3=e_1$, & $x_4=e_4$, \\
& $x_5=-e_7$, & $x_6=-e_5$, & $x_7=e_8$, &  $x_8=\frac{1}{t}e_6$. \\ \hline

\multirow{2}{*}{$\n_{5,3}\oplus\n_{3,1}\to(27A)$} & $x_1=e_5$, & $x_2=-e_3$, & $x_3=e_7$, & $x_4=-e_4$,\\
 & $x_5=e_8$, & $x_6=e_1$, & $x_7=e_2$, & $x_8=\frac{1}{t}(e_6-e_2)$. \\ \hline

\multirow{2}{*}{$N_2^{8,2}\to (27B)$} & $x_1=e_1$, & $x_2=e_2-e_6$, & $x_3=e_4$, & $x_4=e_5$, \\
& $x_5=e_3$, & $x_6=e_7$, & $x_7=e_8$, & $x_8=te_6$.\\ \hline

\multirow{2}{*}{$N_2^{8,3}\to N_5^{8,3}$} & $x_1=t^{-1}(e_2+e_3)$, & $x_2=te_4$, & $x_3=-te_1$, & $x_4=-te_2$, \\
& $x_5=t^{-1}e_5$, & $x_6=e_8$, & $x_7=e_6$, & $x_8=e_7$. \\ \hline

%

\multirow{2}{*}{$N_8^{8,3}\to (27B)$} & $x_1=e_2$, & $x_2=e_3$, & $x_3=-e_1$, & $x_4=-e_4$, \\
& $x_5=e_5$, & $x_6=e_8$, & $x_7=e_7$, & $x_8=\frac{1}{t}e_6$. \\ \hline

\multirow{2}{*}{$N_8^{8,3}\to N_5^{8,3}$} & $x_1=e_1$, & $x_2=e_2$, & $x_3=e_3$, & $x_4=\frac{1}{t}e_4$,\\
 & $x_5=e_5$, & $x_6=e_6$, & $x_7=e_7$, & $x_8=\frac{1}{t}e_8$. \\ \hline

\multirow{2}{*}{$N_8^{8,3}\to(37D)$} & $x_1=e_1$, & $x_2=e_3$, & $x_3=\frac{1}{t}e_2$, & $x_4=te_5-e_2$, \\
& $x_5=e_7$, & $x_6=\frac{1}{t}e_6$, & $x_7=e_8$, & $x_8=te_4$. \\ \hline

\multirow{2}{*}{$\n_{5,1}\oplus\n_{3,1}\to(37B)$} & $x_1=e_3$, & $x_2=e_5$, & $x_3=e_4+e_7$, & $x_4=e_8$, \\
& $x_5=e_1$, & $x_6=-e_2$, & $x_7=e_6$, & $x_8=te_7$. \\ \hline

\multirow{2}{*}{$\n_{5,1}\oplus\n_{3,1}\to(27A)$} & $x_1=e_5$, & $x_2=-e_3$, & $x_3=e_7$, & $x_4=-e_4$, \\
& $x_5=e_8$, & $x_6=e_1$, & $x_7=e_2$, & $x_8=\frac{1}{t}(e_6-e_2)$. \\ \hline

%

\multirow{2}{*}{$N_1^{8,3}\to N_2^{8,3}$} & $x_1=-te_1$, & $x_2=e_2+e_4$, & $x_3=e_3$, & $x_4=e_1+e_5$,\\
 & $x_5=te_4$, & $x_6=-te_6$, & $x_7=e_7$, & $x_8=-e_8$. \\ \hline
 
\multirow{2}{*}{$N_1^{8,3}\to \n_{5,1}\oplus\n_{3,1}$} & $x_1=e_7$, & $x_2=te_8$, & $x_3=e_2$, & $x_4=te_1$,\\
 & $x_5=e_3$, & $x_6=e_6$, & $x_7=e_4$, & $x_8=e_5$. \\ \hline

%

\multirow{2}{*}{$N_3^{8,4}\to(37D)$} & $x_1=e_1$, & $x_2=e_3$, & $x_3=e_4$, & $x_4=-e_2$, \\
& $x_5=e_6$, & $x_6=e_8$, & $x_7=e_7$, & $x_8=\frac{1}{t}e_5$. \\ \hline

\multirow{2}{*}{$N_3^{8,4}\to N_2^{8,4}$} & $x_1=e_1$, & $x_2=e_2$, & $x_3=e_3$, & $x_4=te_4$,\\
 & $x_5=e_5$, & $x_6=e_6$, & $x_7=e_7$, & $x_8=te_8$. \\ \hline

\multirow{2}{*}{$N_5^{8,2}\to N_2^{8,2}$} & $x_1=\frac{1}{t}e_2$, & $x_2=-te_1$, & $x_3=te_3$, & $x_4=e_5$, \\
& $x_5=e_6$, & $x_6=-e_4$ & $x_7=e_7$, & $x_8=e_8$. \\ \hline

\multirow{2}{*}{$N_5^{8,2}\to\n_{5,3}\oplus\n_{3,1}$} & $x_1=e_8$, & $x_2=e_2$, & $x_3=e_4$, & $x_4=e_3$, \\
& $x_5=e_5$, & $x_6=\frac{1}{t}e_7$, & $x_7=e_1+\frac{1}{t}e_5$, & $x_8=e_6$. \\ \hline

\multirow{2}{*}{$N_5^{8,2}\to N_4^{8,2}$} & $x_1=\frac{1}{t}e_1$, & $x_2=te_2$, & $x_3=e_3$, & $x_4=e_4$, \\
& $x_5=e_5$, & $x_6=e_6$, & $x_7=e_7$, & $x_8=e_8$. \\ \hline

\multirow{2}{*}{$N_7^{8,3}\to N_2^{8,3}$} & $x_1=\frac{1}{t}e_5$, & $x_2=-te_1$, & $x_3=e_3$, & $x_4=e_4$, \\
& $x_5=-e_2$, & $x_6=e_8$, & $x_7=-e_7$, & $x_8=e_6$. \\ \hline

\multirow{2}{*}{$N_7^{8,3}\to N_8^{8,3}$} & $x_1=e_1$, & $x_2=-\frac{1}{t}e_3+e_4$, & $x_3=e_2$, & $x_4=-e_5$, \\
& $x_5=-e_3$, & $x_6=-\frac{1}{t}e_7$, & $x_7=e_6$, & $x_8=-e_8$. \\ \hline




\multirow{2}{*}{$N_1^{8,4}\to N_3^{8,4}$} & $x_1=-e_1+\frac{1}{\sqrt{t}}e_2$, & $x_2=-\frac{1}{\sqrt{t}}e_3+e_4$, & $x_3=e_3+\sqrt{t}e_4$, & $x_4=\sqrt{t}e_1+e_2$, \\
& $x_5=-\frac{1}{t}e_6+e_8$, & $x_6=\frac{1}{\sqrt{t}}e_6+\sqrt{t}e_8$, & $x_7=-2e_7$, & $x_8=-2e_5$. \\ \hline

\multirow{2}{*}{$N_3^{8,2}\to N_2^{8,2}$} & $x_1=e_1+e_5$, & $x_2=te_2$, & $x_3=e_6$, & $x_4=e_2+e_4$, \\
& $x_5=-te_1$, & $x_6=-e_3$, & $x_7=te_7$, & $x_8=e_8$. \\ \hline

\multirow{2}{*}{$N_3^{8,2}\to \n_{5,3}\oplus\n_{3,1}$} & $x_1=e_8$, & $x_2=e_5$, & $x_3=-te_4$, & $x_4=e_6$, \\
& $x_5=\frac{1}{t}e_3$, & $x_6=e_7$, & $x_7=e_1$, & $x_8=e_2$. \\ \hline

\multirow{2}{*}{$N_3^{8,2}\to N_4^{8,2}$} & $x_1=-te_1$, & $x_2=e_2$, & $x_3=e_6$, & $x_4=\frac{1}{t}e_3+e_5$, \\
& $x_5=e_4$, & $x_6=e_3$, & $x_7=-e_8$, & $x_8=-e_7+\frac{1}{t}e_8$. \\ \hline




\multirow{2}{*}{$N_3^{8,3}\to N_2^{8,3}$} & $x_1=-e_3$, & $x_2=e_2$, & $x_3=-e_1+te_3$, & $x_4=e_4$, \\
& $x_5=e_5$, & $x_6=e_7$, & $x_7=e_6$, & $x_8=te_8$. \\ \hline

\multirow{2}{*}{$N_3^{8,3}\to N_8^{8,3}$} & $x_1=e_1+e_4$, & $x_2=\frac{1}{t}e_2+e_3$, & $x_3=te_3$, & $x_4=e_5$, \\
& $x_5=te_1$, & $x_6=\frac{1}{t}e_6-e_8$, & $x_7=-e_6$, & $x_8=e_7$. \\ \hline

\multirow{2}{*}{$N_3^{8,3}\to \n_{5,1}\oplus\n_{3,1}$} & $x_1=e_6$, & $x_2=e_7$, & $x_3=e_1$, & $x_4=-e_3$, \\
& $x_5=e_2$, & $x_6=te_8$, & $x_7=e_4$, & $x_8=t(e_5-e_3)$. \\ \hline





\multirow{2}{*}{$N_1^{8,2}\to N_3^{8,2}$} &  $x_1=e_1$, & $x_2=e_2$,&  $x_3=te_3$, & $x_4=e_4-e_6$, \\
& $x_5=e_5+e_3$, & $x_6=te_6$, & $x_7=e_7$, &$x_8=te_8$. \\ \hline

\multirow{2}{*}{$N_1^{8,2}\to N_5^{8,2}$} &  $x_1=-t(e_3+e_6)$, & $x_2=e_2+e_4+e_5$,&  $x_3=i\sqrt{t}(-e_3+e_6)$, & $x_4=i\sqrt{t}(-e_4+e_5)$, \\
& $x_5=e_1+e_3+e_6$, & $x_6=te_2$, & $x_7=te_7$, &$x_8=i\sqrt{t}(e_7+2e_8)$. \\ \hline

%

\multirow{2}{*}{$N_{6}^{8,3}\to N_{3}^{8,3}$} & $x_1=e_5$, & $x_2=e_4$, & $x_3=-te_1$, & $x_4=te_2$, \\
& $x_5=e_3$, & $x_6=-e_7$, & $x_7=te_8$, & $x_8=-t^2e_6$. \\ \hline

\multirow{2}{*}{$N_{6}^{8,3}\to N_{7}^{8,3}$} & $x_1=e_1+e_4$, & $x_2=te_1$, & $x_3=te_2$, & $x_4=-e_3+e_5$, \\
& $x_5=-te_5$, & $x_6=-te_8$, & $x_7=te_6$, & $x_8=-te_7$. \\ \hline

\multirow{2}{*}{$N_{6}^{8,3}\to N_{1}^{8,3}$} & $x_1=-te_1$, & $x_2=e_3$, & $x_3=e_2$, & $x_4=e_4$, \\
& $x_5=-te_5$, & $x_6=-te_7$, & $x_7=-e_8$, & $x_8=-te_6$. \\ \hline

\multirow{2}{*}{$N_{10}^{8,3}\to N_{3}^{8,3}$} & $x_1=e_5$, & $x_2=e_2+2e_4$, & $x_3=\frac{i}{\sqrt{t}}(-e_1+e_3)$, & $x_4=-\frac{i}{\sqrt{t}}e_2$, \\
& $x_5=e_1+e_3$, & $x_6=-2e_8$, & $x_7=\frac{i}{\sqrt{t}}(e_6-e_7)$, & $x_8=-\frac{1}{t}(e_6+e_7)$. \\ \hline

\multirow{2}{*}{$N_{10}^{8,3}\to N_{7}^{8,3}$} & $x_1=e_2+\frac{1}{2}e_4$, & $x_2=i\sqrt{t}e_3+e_5$, & $x_3=e_1+e_3-\frac{1}{i\sqrt{t}}e_5$, & $x_4=\frac{i\sqrt{t}}{2}e_4$, \\
& $x_5=te_1$, & $x_6=\frac{i\sqrt{t}}{2}e_7+\frac{1}{2}e_8$, & $x_7=-e_6+\frac{1}{2}e_7-\frac{1}{2i\sqrt{t}}e_8$, & $x_8=-\frac{t}{2}e_7-\frac{i\sqrt{t}}{2}e_8$. \\ \hline

\multirow{2}{*}{$N_{11}^{8,3}\to N_{10}^{8,3}$} & $x_1=e_3$, & $x_2=e_4$, & $x_3=e_5$, & $x_4=e_1$, \\
& $x_5=te_2$, & $x_6=e_6$, & $x_7=e_7$, & $x_8=te_6$. \\ \hline

\multirow{2}{*}{$N_{11}^{8,3}\to N_{6}^{8,3}$} & $x_1=e_2+e_3$, & $x_2=e_1+e_4$, & $x_3=i\sqrt{t}(e_2-e_3)$, & $x_4=i\sqrt{t}(e_4-e_1)$, \\
& $x_5=-e_5$, & $x_6=-e_6+e_8$, & $x_7=-2i\sqrt{t}e_7$, & $x_8=i\sqrt{t}(e_6+e_8)$. \\ \hline


\multirow{2}{*}{$N_{4}^{8,3}\to N_{11}^{8,3}$} & $x_1=e_1+e_2+e_4$, & $x_2=2\sqrt{t}(e_2-e_1)$, & $x_3=-e_3+e_5$, & $x_4=-2e_4$, \\
& $x_5=\sqrt{t}(e_3+e_5)$, & $x_6=4\sqrt{t}e_6$, & $x_7=2\sqrt{t}(-e_7+e_8)$, & $x_8=2(e_8+e_7)$. \\ \hline


\multirow{2}{*}{$N_{9}^{8,3}\to N_{4}^{8,3}$} & $x_1=e_4$, & $x_2=e_5$, & $x_3=te_4$, & $x_4=\frac{1}{t}e_5$, \\
& $x_5=te_3$, & $x_6=e_6$, & $x_7=e_8$, & $x_8=e_7$. \\ \hline

\end{longtable}
\end{spacing}
}

\normalsize

Finally, the Hasse diagram of degenerations is given by:

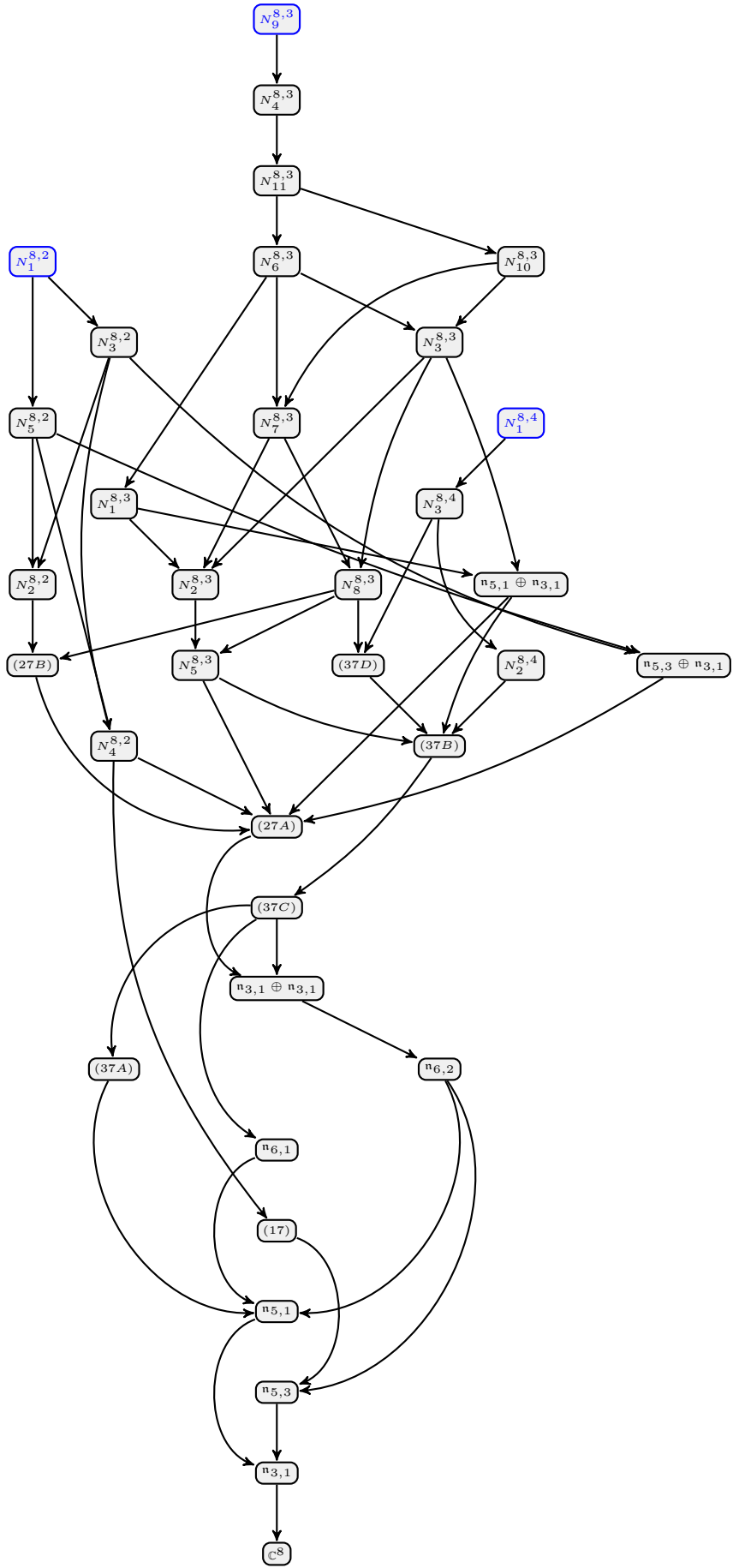
\fontsize{5pt}{5pt}
{
 
\begin{center}

\begin{tikzpicture}[->,>=stealth',shorten >=0.01cm,auto,node distance=1.25cm,
                    thick,main node/.style={rectangle,draw,fill=gray!12,rounded corners=1.5ex,font=\sffamily \bf \bfseries },
                    blue node/.style={rectangle,draw, color=blue,fill=gray!12,rounded corners=1.5ex,font=\sffamily \bf \bfseries },
                    rigid node/.style={rectangle,draw,fill=black!20,rounded corners=1.5ex,font=\sffamily \tiny \bfseries },style={draw,font=\sffamily \scriptsize \bfseries }]
                    
\node (200)   {};

\node (190) [below          of=200]      {};
\node (180) [below         of=190]      {};
\node (170) [below         of=180]      {};
\node (160) [below         of=170]      {};
\node (150) [below         of=160]      {};
\node (140) [below         of=150]      {};
\node (130) [below         of=140]      {};
\node (120) [below         of=130]      {};
\node (110) [below         of=120]      {};
\node (100) [below         of=110]      {};
\node (90) [below         of=100]      {};
\node (80) [below         of=90]      {};
\node (70) [below         of=80]      {};
\node (60) [below         of=70]      {};
\node (50) [below         of=60]      {};
\node (40) [below         of=50]      {};
\node (30) [below         of=40]      {};
\node (20) [below         of=30]      {};
\node (10) [below         of=20]      {};
\node (00) [below         of=10]      {};

\node (201)  [right of =200]                      { };
\node (202)  [right of =201]                      { };
\node (191)  [right of =190]                      { };
\node (192)  [right of =191]                      {};
\node (181)  [right of =180]                      { };
\node (182)  [right of =181]                      {};
\node (171)  [right of =170]                      { };
\node (172)  [right of =171]                      {};
\node (173)  [right of =172]                      { };
\node (174)  [right of =173]                      { };
\node (161)  [right of =160]                      { };
\node (162)  [right of =161]                      {};
\node (163)  [right of =162]                      { };
\node (151)  [right of =150]                      {};
\node (152)  [right of =151]                      {};
\node (153)  [right of =152]                      { };
\node (154)  [right of =153]                      { };
\node (141)  [right of =140]                      {};
\node (142)  [right of =141]                      { };
\node (143)  [right of =142]                      {};
\node (131)  [right of =130]                      {};
\node (132)  [right of =131]                      { };
\node (133)  [right of =132]                      {};
\node (134)  [right of =133]                      {};
\node (135)  [right of =134]                      {};
\node (136)  [right of =135]                      {};
\node (137)  [right of =136]                      {};
\node (121)  [right of =120]                      {};
\node (122)  [right of =121]                      { };
\node (123)  [right of =122]                      {};
\node (124)  [right of =123]                      {};
\node (125)  [right of =124]                      {};
\node (126)  [right of =125]                      {};
\node (127)  [right of =126]                      {};
\node (128)  [right of =127]                      {};
\node (111)  [right of =110]                      {};
\node (112)  [right of =111]                      { };
\node (113)  [right of =112]                      {};
\node (101)  [right of =100]                      {};
\node (102)  [right of =101]                      { };
\node (91)  [right of =90]                      {};
\node (92)  [right of =91]                      { };
\node (81)  [right of =80]                      {};
\node (82)  [right of =81]                      { };
\node (71)  [right of =70]                      {};
\node (72)  [right of =71]                      { };
\node (73)  [right of =72]                      {};
\node (61)  [right of =60]                      {};
\node (62)  [right of =61]                      { };
\node (51)  [right of =50]                      {};
\node (52)  [right of =51]                      { };
\node (41)  [right of =40]                      {};
\node (42)  [right of =41]                      { };
\node (31)  [right of =30]                      {};
\node (32)  [right of =31]                      { };
\node (21)  [right of =20]                      {};
\node (22)  [right of =21]                      { };
\node (11)  [right of =10]                      {};
\node (12)  [right of =11]                      { };

\node [blue node] (839)  [right of =201]                      {$N_9^{8,3}$};

\node [main node] (834)  [right of =191]                      {$N_4^{8,3}$ };

\node [main node] (8311)  [right of =181]                      {$N_{11}^{8,3}$};

\node [blue node] (821)  [left of =170]                      {$N_1^{8,2}$};
\node [main node] (836)  [right of =171]                      {$N_6^{8,3}$ };
\node [main node] (8310)  [right of =174]                      {$N_{10}^{8,3}$};

\node [main node] (823)  [left of =161]                      {$N_3^{8,2}$};
\node [main node] (833)  [right of =163]                      {$N_3^{8,3}$};

\node [main node] (825)  [left of =150]                      {$N_5^{8,2}$};
\node [main node] (837)  [right of =151]                      {$N_7^{8,3}$};
\node [blue node] (841)  [right of =154]                      {$N_1^{8,4}$};

\node [main node] (831)  [left of =141]                      {$N_1^{8,3}$};
\node [main node] (843)  [right of =143]                      {$N_3^{8,4}$};

\node [main node] (822)  [left of =130]                      {$N_2^{8,2}$};
\node [main node] (832)  [right of =130]                      {$N_2^{8,3}$};
\node [main node] (838)  [right of =132]                      {$N_8^{8,3}$};
\node [main node] (n51n31)  [right of =134]                      {$\n_{5,1}\oplus\n_{3,1}$};

\node [main node] (27b)  [left of =120]                      {$(27B)$};
\node [main node] (835)  [right of =120]                      {$N_5^{8,3}$};
\node [main node] (37d)  [right of =122]                      {$(37D)$};
\node [main node] (842)  [right of =124]                      {$N_2^{8,4}$};
\node [main node] (n53n31)  [right of =126]                      {$\n_{5,3}\oplus\n_{3,1}$};

\node [main node] (824)  [left of =111]                      {$N_4^{8,2}$};
\node [main node] (37b)  [right of =113]                      {$(37B)$};

\node [main node] (27a)  [right of =101]                      {$(27A)$};

\node [main node] (37c)  [right of =91]                      {$(37C)$};

\node [main node] (n31n31)  [right of =81]                      {$\n_{3,1}\oplus\n_{3,1}$};

\node [main node] (37a)  [left of =71]                      {$(37A)$};
\node [main node] (n62)  [right of =73]                      {$\n_{6,2}$};

\node [main node] (n61)  [right of =61]                      {$\n_{6,1}$};

\node [main node] (17)  [right of =51]                      {$(17)$};

\node [main node] (n51)  [right of =41]                      {$\n_{5,1}$};

\node [main node] (n53)  [right of =31]                      {$\n_{5,3}$};

\node [main node] (n31)  [right of =21]                      {$\n_{3,1}$};

\node [main node] (c)  [right of =11]                      {$\C^8$};

\path[every node/.style={font=\sffamily\tiny}]

(n31)  edge [bend right=0, color=black] node{}  (c)
(n53)  edge [bend right=0, color=black] node{}  (n31)
(n51)  edge [bend right=70, color=black] node{}  (n31)
(17)  edge [bend right=-70, color=black] node{}  (n53)
(n61)  edge [bend right=70, color=black] node{}  (n51)
(37a)  edge [bend right=60, color=black] node{}  (n51)
(n62)  edge [bend right=-60, color=black] node{}  (n51)
(n62)  edge [bend right=-60, color=black] node{}  (n53)
(n31n31)  edge [bend right=0, color=black] node{}  (n62)
(37c)  edge [bend right=50, color=black] node{}  (37a)
(37c)  edge [bend right=0, color=black] node{}  (n31n31)
(37c)  edge [bend right=60, color=black] node{}  (n61)

(27a)  edge [bend right=70, color=black] node{}  (n31n31)

(824)  edge [bend right=0, color=black] node{}  (27a)
(824)  edge [bend right=20, color=black] node{}  (17)

(37b)  edge [bend right=-10, color=black] node{}  (37c)

(27b)  edge [bend right=40, color=black] node{}  (27a)

(835)  edge [bend right=10, color=black] node{}  (37b)
(835)  edge [bend right=0, color=black] node{}  (27a)

(37d)  edge [bend right=0, color=black] node{}  (37b)

(842)  edge [bend right=0, color=black] node{}  (37b)

(n53n31)  edge [bend right=-10, color=black] node{}  (27a)

(822)  edge [bend right=0, color=black] node{}  (27b)

(832)  edge [bend right=0, color=black] node{}  (835)

(838)  edge [bend right=0, color=black] node{}  (27b)
(838)  edge [bend right=0, color=black] node{}  (835)
(838)  edge [bend right=0, color=black] node{}  (37d)

(n51n31)  edge [bend right=10, color=black] node{}  (37b)
(n51n31)  edge [bend right=0, color=black] node{}  (27a)

(831)  edge [bend right=0, color=black] node{}  (832)
(831)  edge [bend right=0, color=black] node{}  (n51n31)

(843)  edge [bend right=0, color=black] node{}  (37d)
(843)  edge [bend right=30, color=black] node{}  (842)

(825)  edge [bend right=0, color=black] node{}  (822)
(825)  edge [bend right=4, color=black] node{}  (n53n31)
(825)  edge [bend right=0, color=black] node{}  (824)

(837)  edge [bend right=0, color=black] node{}  (832)
(837)  edge [bend right=0, color=black] node{}  (838)

(841)  edge [bend right=0, color=black] node{}  (843)

(823)  edge [bend right=0, color=black] node{}  (822)
(823)  edge [bend right=15, color=black] node{}  (824)
(823)  edge [bend right=15, color=black] node{}  (n53n31)

(833)  edge [bend right=0, color=black] node{}  (832)
(833)  edge [bend right=10, color=black] node{}  (838)
(833)  edge [bend right=-5, color=black] node{}  (n51n31)

(821)  edge [bend right=0, color=black] node{}  (823)
(821)  edge [bend right=0, color=black] node{}  (825)

(836)  edge [bend right=0, color=black] node{}  (837)
(836)  edge [bend right=0, color=black] node{}  (833)
(836)  edge [bend right=0, color=black] node{}  (831)

(8310)  edge [bend right=30, color=black] node{}  (837)
(8310)  edge [bend right=0, color=black] node{}  (833)

(8311)  edge [bend right=0, color=black] node{}  (8310)
(8311)  edge [bend right=0, color=black] node{}  (836)

(834)  edge [bend right=0, color=black] node{}  (8311)

(839)  edge [bend right=0, color=black] node{}  (834);

\end{tikzpicture}

\end{center}

}

\normalsize

With all this, we obtain:

\begin{theorem}
The irreducible components of the variety $\mathcal{N}_8^2$ are:
\begin{itemize}
\item $C_1=\overline{O(N_1^{8,2})}$,
\item $C_2=\overline{O(N_9^{8,3})}$,
\item $C_3=\overline{O(N_1^{8,4})}$.
\end{itemize}
Moreover, the Lie algebras $N_1^{8,2}$, $N_9^{8,3}$ and $N_1^{8,4}$ are rigid in $\mathcal{N}_8^2$.
\end{theorem}

%

\bigskip

%
%
%
%
%
%
%
%
%
%

{\bf Acknowledgements.}
This work was started during a research stay at FaMAF - Universidad Nacional de C\'ordoba, supported by MINEDUC-UA project - code ANT 1855. The author is also supported by ``Fondo Puente de Investigaci\'on de Excelencia'' FPI-18-02, Universidad de Antofagasta.

\end{document}